\documentclass[11pt]{article}
\usepackage{color}
\usepackage{graphicx}
\usepackage{amsfonts,amsmath,latexsym,amssymb}
\usepackage{amsthm}
\newtheorem{theorem}{Theorem}[section]
\newtheorem{corollary}{Corollary}[section]

\newtheorem{proposition}{Proposition}[section]
\newtheorem{remark}{Remark}[section]

\numberwithin{equation}{section}
\begin{document}
\title{Improved Young and Heinz inequalities with the Kantorovich constant}
\author{Wenshi Liao$^a$\footnote{Corresponding author. E-mail: liaowenshi@gmail.com, jlwu678@163.com.} and Junliang Wu$^a$\\
$^a${\small College of Mathematics and Statistics, Chongqing University, Chongqing 401331, P.R. China}}

\date{}
\maketitle
{\bf Abstract.} In this article, we study the further refinements and reverses
of the Young and Heinz inequalities with the Kantorovich constant.
These modified inequalities are used to establish corresponding operator inequalities on Hilbert
space and Hilbert-Schmidt norm inequalities.
\vspace{3mm}

\vspace{3mm}

%{\bf PACS numbers :} 03.65.Ta, 03.67.-a and 02.30.Tb
%\vspace{3mm}

\vspace{3mm}

%%%%%%%%%%%%%%%%%%%%%%%%%%%%%%%%%%%%%%%%%%%%%%%%%%%%%%%%%%%%%%%%%%%%%%%%%%%%%%%%%%%%%%%%%%%%%%%%%%%%%%%%%%%%%%%%%%%%%%%%%%%%%%%%%%%%%%%%%%%%%%%%%%%%%%%%%%%%%%%
%%%%%%%%%%%%%%%%%%%%%%%%%%%%%%%%%%%%%%%%%%%%%%%%%%%%%%%%%%%%%%%%%%%%%%%%%%%%%%%%%%%%%%%%%%%%%%%%%%%%%%%%%%%%%%%%%%%%%%%%%%%%%%%%%%%%%%%%%%%%%%%%%%%%%%%%%%%%%%%
%%%%%%%%%%%%%%%%%%%%%%%%%%%%%%%%%%%%%%%%%%%%%%%%%%%%%%%%%%%%%%%%%%%%%%%%%%%%%%%%%%%%%%%%%%%%%%%%%%%%%%%%%%%%%%%%%%%%%%%%%%%%%%%%%%%%%%%%%%%%%%%%%%%%%%%%%%%%%%%
%%%%%%%%%%%%%%%%%%%%%%%%%%%%%%%%%%%%%%%%%%%  Section1  %%%%%%%%%%%%%%%%%%%%%%%%%%%%%%%%%%%%%%%%%%%%%%%%%%%%%%%%%%%%%%%%%%%%%%%%%%%%%%%%%%%%%%%%%%%%%%%%%%%%%%%%
%%%%%%%%%%%%%%%%%%%%%%%%%%%%%%%%%%%%%%%%%%%%%%%%%%%%%%%%%%%%%%%%%%%%%%%%%%%%%%%%%%%%%%%%%%%%%%%%%%%%%%%%%%%%%%%%%%%%%%%%%%%%%%%%%%%%%%%%%%%%%%%%%%%%%%%%%%%%%%%
%%%%%%%%%%%%%%%%%%%%%%%%%%%%%%%%%%%%%%%%%%%%%%%%%%%%%%%%%%%%%%%%%%%%%%%%%%%%%%%%%%%%%%%%%%%%%%%%%%%%%%%%%%%%%%%%%%%%%%%%%%%%%%%%%%%%%%%%%%%%%%%%%%%%%%%%%%%%%%%
%%%%%%%%%%%%%%%%%%%%%%%%%%%%%%%%%%%%%%%%%%%%%%%%%%%%%%%%%%%%%%%%%%%%%%%%%%%%%%%%%%%%%%%%%%%%%%%%%%%%%%%%%%%%%%%%%%%%%%%%%%%%%%%%%%%%%%%%%%%%%%%%%%%%%%%%%%%%%%%
\section{Introduction}
The well-known Young inequality for two scalars is the weighted arithmetic-geometric mean inequality, which was
attributed to the English mathematician William Henry Young (1863-1942). The inequality states that if $a,b>0$
and $0\le v\le1$, then
\begin{equation}
\label{eq1.1}
(1-v )a+v b\ge a^{1-v }b^v
\end{equation}
with equality if and only if $a=b$. If $v=\frac{1}{2}$, we obtain the fundamental arithmetic-geometric mean
inequality $2\sqrt{ab}\le a+b$. If $v>1$ or $v<0$, then the reverse inequality of \eqref{eq1.1}
$$(1-v )a+v b\le a^{1-v }b^v$$
holds, the proof can be found in \cite{Bakherad}.

The Heinz mean, defined as
\[
H_{v}(a,b)=\frac{a^{v}b^{1-v}+a^{1-v}b^{v}}{2}
\]
for $a,b>0$ and $0\le v\le1$, interpolates between the arithmetic mean and geometric mean.
It is easy to see that the Heinz mean is convex as a function of $v$ on the interval $[0,1]$,
attains minimum at $v=1/2$, and attains maximum at $v=0$ and $v=1$, so
\begin{equation}
\label{eq2}
\sqrt{ab}=H_{\frac{1}{2}}(a,b)\le H_{v}(a,b)\le H_{1}(a,b)=\frac{a+b}{2}.
\end{equation}
 Moreover, $H_{v}(a,b)$ is symmetric about the point $v=1/2$, that is, $H_v(a,b)=H_{1-v}(a,b)$.

%Heron means [1] are the family of means defined as a convex combination of the arithmetic and the geometric mean, that is
%\[F_{\alpha}(a,b)=(1-\alpha)A_{\frac{1}{2}}(a,b)+\alpha G_{\frac{1}{2}}(a,b) \]
%for $0\leq\alpha\leq 1$. It is evident that
%    \begin{equation}
%        \label{important}
%      \sqrt{ab}\leq F_{\alpha}(a,b)\leq \frac{a+b}{2}.
%\end{equation}

Kittaneh and Manasrah \cite{Kittaneh1, Kittaneh2} improved the Young inequality \eqref{eq1.1}, and obtained the
following relations:
\begin{equation}
\label{km1}
r(\sqrt a - \sqrt b )^2\le (1 - v )a + v b-a^{1 - v }b^v \le  R(\sqrt a - \sqrt b )^2,
\end{equation}
%\begin{equation}
%\label{km2}
%2r(\sqrt a - \sqrt b )^2\le a + b-\left(a^{v}b^{1-v}+a^{1 - v }b^v\right)\le  2R(\sqrt a - \sqrt b )^2,
%\end{equation}
%\begin{equation}
%\label{km3}
% 2r( a - b )^2\le (a + b)^2-\left(a^{v}b^{1-v}+a^{1 - v }b^v\right)^2 \le 2R( a - b )^2,
%\end{equation}
where $a,b>0$, $v\in [0,1]$, $r=\min\{v, 1-v\}$ and $R=\max\{v, 1-v\}$.

%\begin{equation}
%\label{th111}
%r(\sqrt a - \sqrt b )^2\le\left( {1 - v} \right)a + v b-a^{1 - v}b^v  \le R(\sqrt a - \sqrt b )^2
%\end{equation}
Wu and Zhao \cite{Wu} presented further improvements of \eqref{km1} that
\begin{equation}
\label{wu1}
\left( {1 - v} \right)a + v b\ge r(\sqrt a - \sqrt b )^2+ {\rm K}(\sqrt h ,2)^{r_1}a^{1 - v}b^v,
\end{equation}
\begin{equation}
\label{wu2}
\left( {1 - v} \right)a + v b\le R(\sqrt a - \sqrt b )^2+  {\rm K}(\sqrt h ,2)^{-r_1}a^{1 - v}b^v,
\end{equation}
where $h=\frac{b}{a}$, ${\rm K}(\sqrt {h} ,2)=\frac{(\sqrt {h}+1)^2}{4\sqrt {h}}$,
 ${r_1} = \min \left\{ {2r,1 - 2r} \right\}$.
Note that ${\rm K}(t,2)=\frac{(t+1)^2}{4t}$ is the classical
 Kantorovich constant which has properties ${\rm K}(1,2)=1,$ ${\rm K}(t,2)={\rm K}\left(
{\frac{1}{t},2} \right)\ge 1(t>0)$ and ${\rm K}(t,2)$ is monotone
increasing on $[1,\infty )$ and monotone decreasing on $(0,1]$.

%Inspired by the Furuichi's refinement of arithmetic-geometric mean with Specht's ratio (see \cite{Furuichi}),
%we obtained a reverse ratio forms of inequality \eqref{wu1} via Kantorovich constant in \cite{Liao} as follows
%\begin{equation}
%\label{liao1}
%\left( {1 - v} \right)a + v b \le r(\sqrt a - \sqrt b )^2 +{\rm K}(\sqrt h ,2)^{R_1}a^{1 - v}b^v
%\end{equation}
%where ${R_1} = \max \left\{ {2r,1 - 2r} \right\}$.

Recently, Zhao and Wu obtained the refinements and reverses of
 Young inequality and improved inequalities \eqref{km1} in the
following forms:

\begin{proposition}\cite{Zhao}
Let $a,b$ be two nonnegative real numbers and $v\in (0,1)$.

$\rm (I)$ If $0<v\leq \frac{1}{2}$, then
\begin{equation}
   \label{p1}
(1-v)a+v b \ge a^{1-v}b^v +v(\sqrt{a}-\sqrt{b})^2+r_1(\sqrt[4]{ab}-\sqrt{a})^2,
\end{equation}
\begin{equation}
   \label{p2}
(1-v)a+v b\le a^{1-v}b^v+(1-v)(\sqrt{a}-\sqrt{b})^2-r_1(\sqrt[4]{ab}-\sqrt{b})^2,
\end{equation}

 $\rm (II)$ if $\frac{1}{2}<v< 1$, then
 \begin{equation}
   \label{p3}
(1-v)a+v b\ge a^{1-v}b^v+(1-v)(\sqrt{a}-\sqrt{b})^2+ r_1(\sqrt[4]{ab}-\sqrt{b})^2 ,
\end{equation}
\begin{equation}
   \label{p4}
(1-v)a+v b\le a^{1-v}b^v+v(\sqrt{a}-\sqrt{b})^2-r_1(\sqrt[4]{ab}-\sqrt{a})^2,
\end{equation}
where $r=\min\{{v,1-v}\}$ and $r_1=\min\{{2r,1-2r}\}$.
\end{proposition}
%And they also obtained the reverse forms of the inequalities \eqref{p1} and \eqref{p2},
% which are more precise than the second inequality in \eqref{km1}.
%\begin{proposition} \cite{Zhao}
%Let $a,b$ be two nonnegative real numbers and $v\in (0,1)$.
%
%$\rm (I)$ If $0<v\leq \frac{1}{2}$, then
%
%
% $\rm (II)$ if $\frac{1}{2}<v< 1$, then
%
%\end{proposition}

%Throughout this paper, we set
%\[
%a\nabla _v b=(1-v )a+vb,~a\nabla b=\frac{a+b}{2}
%\]
%and
%\[
%a\nabla _v b=(1-v )a+vb,~a\nabla b=\frac{a+b}{2}
%\]
Let $\mathcal{B}(\mathcal{H})$ be the $C^\ast $-algebra of all bounded linear operators on a
complex separable Hilbert space $(\mathcal{H}, \langle \cdot,\cdot\rangle)$. $I$ stands for
the identity operator. ${\mathcal{B}}^{++}(\mathcal{H})$ denotes the cone of all positive invertible
operators on $\mathcal{H}$. As a matter of convenience, we use the following
notations to define the weighted arithmetic mean and geometric mean for operators:
\[
A\nabla _v B=(1-v)A+v B,~A\# _v B=A^{\frac{1}{2}}(A^{-\frac{1}{2}}BA^{-\frac{1}{2}})^v A^{\frac{1}{2}},~
\]
where
$A,B\in {\mathcal{B}}^{++}(\mathcal{H})$ and $v \in [0,1]$.
When $v =\frac{1}{2}$, we write  $A\nabla B$ and $A\# B$
for brevity, respectively.

%It is well known that the famous weighted A-G-H mean
%inequalities hold
%\begin{equation}
%\label{eqagh}
%a\nabla _\mu b\ge a^{1-\mu }b^\mu \ge a!_\mu b
%\end{equation}
%for $a,b>0$ and $\mu\in[0,1]$ with equalities if and only if $a=b$. The first inequality of (\ref{eqagh}) is
%the classical Young inequality.

An operator version of the Young inequality proved in \cite{Furuta}
says that if $A, B \in{\mathcal{B}}^{++}(\mathcal{H}) $ and $v\in[0,1]$, then
\[
A\nabla _v B\ge A\#_{v}B.
\]

The Heinz operator mean is defined by
\[H_{v}(A,B)=\frac{A\#_{v}B+A\#_{1-v}B}{2},\]
for $A,B\in {\mathcal{B}}^{++}(\mathcal{H})$ and $0\leq v\leq1$.

It is easy to see that  the Heinz operator mean interpolates
the arithmetic-geometric operator mean inequality:
\begin{equation}
   \label{heinz}
A\#B\le H_{v}(A,B)\le A\nabla B.
\end{equation}
The equation \eqref{heinz} are called the Heinz
operator inequalities (See \cite{Kittaneh3, Kittaneh4}).

The first difference-type improvement of the matrix Young inequality
 is due to Kittaneh and Manasrah
\cite{Kittaneh2} extending (1.2) to matrices:
\begin{equation}
   \label{km4}
r(A\nabla B-A\#B)\le A\nabla_v B-A\#_vB\le R(A\nabla B-A\#B)
\end{equation}
holds for positive definite matrices $A$ and $B$ and $0\leq v\leq1$,
where $r=\min\{v, 1-v\}$ and $R=\max\{v, 1-v\}$, which of course
remain valid for Hilbert space operators by a standard approximation argument.

Note that the first inequality in \eqref{km4} was independently
established for positive operators $A$ and $B$ by Furuichi in \cite{Furuichi}.

The ratio-type improvements of the Young inequality are referred to
\cite{Furuichi, Furuichi1, Tominaga, Wu, Zuo}.

The operator versions of \eqref{wu1} and \eqref{wu2} were presented in \cite{Wu}.

%The operator versions of \eqref{wu1} and \eqref{wu2} in \cite{Wu}, were shown that
%\begin{equation}
%\label{eq20}
% \left( {1 - v } \right)A+ v B \ge 2r(A\nabla B - A\# B)+{\rm K}(\sqrt h ,2)^{r_1}A\# _v B
%\end{equation}
%and
%\begin{equation}
%\label{eq20}
%\left( {1 - v } \right)A+ v B\le 2R(A\nabla B - A\# B)+{\rm K}(\sqrt h ,2)^{-r_1}A\# _v B
%\end{equation}
%for $A, B \in{\mathcal{B}}^{++}(\mathcal{H}) $ and $v\in[0,1]$, where $r=\min\{{v,1-v}\}$, $R=\max\{{v,1-v}\}$ and $r_1=\min\{{2r,1-2r}\}$.
%\begin{equation}
%\label{eq20}
%\left( {1 - v } \right)A+ v B\le 2r(A\nabla B - A\# B)+{\rm K}(\sqrt h ,2)^{R_1}A\# _v B
%\end{equation}

Zhao and Wu \cite{Zhao} also extended inequalities \eqref{p1}-\eqref{p4}
to positive invertible operators and improved \eqref{km4}, which were shown as
\begin{proposition}\cite{Zhao}
Let $A,B\in {\mathcal{B}}^{++}(\mathcal{H})$ and $v\in (0,1)$.

$\rm (I)$ If $0<v\leq \frac{1}{2}$, then
   \begin{equation}
       \label{p5}
        A\nabla_v B\ge 2v(A\nabla B-A\#B)+r_1(A\#B-2A\#_{\frac{1}{4}}B+A)+A\#_vB,
   \end{equation}
   \begin{equation}
       \label{p6}
        A\nabla_v B\leq 2(1-v)(A\nabla B-A\#B)-r_1(A\#B-2A\#_{\frac{3}{4}}B+B)+A\#_vB,
     \end{equation}

  $\rm (II)$ if $\frac{1}{2}<v< 1$, then
  \begin{equation}
       \label{p7}
         A\nabla_v B\ge 2(1-v)(A\nabla B-A\#B)+r_1(A\#B-2A\#_{\frac{3}{4}}B+B)+A\#_vB,
    \end{equation}
  \begin{equation}
       \label{p8}
         A\nabla_v B\leq 2v(A\nabla B-A\#B)-r_1(A\#B-2A\#_{\frac{1}{4}}B+A)+A\#_vB,
  \end{equation}
  where $r=\min\{v, 1-v\}$ and $r_1=\min\{2r, 1-2r\}$.
\end{proposition}

In this paper, we are concerned with several improvements of the Young
and Heinz inequalities via the Kantorovich constant. In Section 2, we
present the whole series of refinements and reverses of the scalars Young
inequality which will help us to derive Heinz inequalities. In Section 3,
we extend inequalities proved in Section 2 from the scalars setting to a
Hilbert space operator setting. In Section 4, the Hilbert-Schmidt norm
inequalities are established.

\section{Scalars inequalities}

In this section, we mainly present the direct refinements and reverses of the Young inequality
for two positive numbers $a, b$. When $v=0$ and $v=1$, the Young inequality is trivial. We will
 study the case $v\in(0, 1)$.

\begin{theorem}\label{th21}
 Let $a, b>0$ and $v\in(0, 1)$. Then

$\rm (I)$ If $0<v\leq \frac{1}{2}$, then
\begin{equation}
\label{th211}
  (1-v)a+vb\geq v(\sqrt{a}-\sqrt{b})^2+r_1(\sqrt[4]{ab}-\sqrt{a})^2+\mathrm{K}(\sqrt[4]{h}, 2)^{\hat{r}_1}a^{1-v}b^v,
\end{equation}

$\rm (II)$ if $\frac{1}{2}<v< 1$, then
\begin{equation}
\label{th212}
  (1-v)a+vb\geq (1-v)(\sqrt{a}-\sqrt{b})^2+ r_1(\sqrt[4]{ab}-\sqrt{b})^2+\mathrm{K}(\sqrt[4]{h}, 2)^{\hat{r}_1}a^{1-v}b^v,
\end{equation}
where $h=\dfrac{b}{a}$, $r=\min\{v, 1-v\}$, $r_1=\min\{2r, 1-2r\}$ and $\hat{r}_1=\min\{2r_1, 1-2r_1\}$.
%where $h=\dfrac{b}{a}$, $r_1=\min\{4v, 1-4v\}$, $r_2=\min\{2-4v,4v-1\}$, $r_3=\min\{4v-2, 3-4v\}$ and $r_4=\min\{4-4v,4v-3\}$ .
\end{theorem}

\begin{proof}
The proof of the inequality \eqref{th212} is similar to that of \eqref{th211}. Thus,
we only need to prove the inequality \eqref{th211}.

If $v=\frac{1}{4}$ and $v=\frac{1}{2}$, the inequality \eqref{th211} becomes equality.

By the inequality \eqref{wu1}, if $0<v<\frac{1}{4}$, then we get
\begin{align*}
(1-v)a+vb&-v(\sqrt{a}-\sqrt{b})^2\\
&=2v\sqrt{ab}+(1-2v)a\\
&\geq 2v(\sqrt[4]{ab}-\sqrt{a})^2+\mathrm{K}(\sqrt[4]{h}, 2)^{\min\{4v, 1-4v\}}a^{1-v}b^v,
\end{align*}
if $\frac{1}{4}<v< \frac{1}{2}$, then we get
\begin{align*}
(1-v)a+vb&-v(\sqrt{a}-\sqrt{b})^2\\
&=2v\sqrt{ab}+(1-2v)a\\
&\geq (1-2v)(\sqrt[4]{ab}-\sqrt{a})^2+\mathrm{K}(\sqrt[4]{h}, 2)^{\min\{2-4v,4v-1\}}a^{1-v}b^v.
\end{align*}
So we conclude that
\[ (1-v)a+vb\geq v(\sqrt{a}-\sqrt{b})^2+r_1(\sqrt[4]{ab}-\sqrt{a})^2+\mathrm{K}(\sqrt[4]{h}, 2)^{\hat{r}_1}a^{1-v}b^v.\]
This completes the proof.
\end{proof}

\begin{remark}
By the properties of Kantorovich constant, \eqref{th211} and \eqref{th212} are better than \eqref{p1} and \eqref{p3}, respectively.
As a direct consequence of Theorem \ref{th21}, we have the following inequalities with respect to the Heinz mean:
%If $0<v\leq \frac{1}{2}$, then
\begin{equation}
\label{rem1}
  \frac{a+b}{2}\geq r(\sqrt{a}-\sqrt{b})^2+\frac{1}{2}r_1\left[(\sqrt[4]{ab}-\sqrt{a})^2+(\sqrt[4]{ab}-\sqrt{b})^2\right]+\mathrm{K}(\sqrt[4]{h}, 2)^{\hat{r}_1}H_v(a,b).
\end{equation}

%If $\frac{1}{2}<v< 1$, then
%\begin{equation}
%\label{eqlem}
% \frac{a+b}{2}\geq (1-v)(\sqrt{a}-\sqrt{b})^2+\frac{1}{2}r_1\left[(\sqrt[4]{ab}-\sqrt{a})^2+(\sqrt[4]{ab}-\sqrt{b})^2\right]+\mathrm{K}(\sqrt[4]{h}, 2)^{\hat{r}_1}H_v(a,b).
%\end{equation}
\end{remark}

\begin{corollary}\label{co21}
 Let $a, b>0$ and $v\in(0, 1)$. Then

$\rm (I)$ If $0<v\leq \frac{1}{2}$, then
\begin{equation}
\label{co211}
  ((1-v)a+vb)^2\geq v^2(a-b)^2+r_1(\sqrt{ab}-a)^2+\mathrm{K}(\sqrt{h}, 2)^{\hat{r}_1}\left(a^{1-v}b^v\right)^2.
\end{equation}

$\rm (II)$ If $\frac{1}{2}<v<1$, then
\begin{equation}
\label{co212}
  ((1-v)a+vb)^2\geq (1-v)^2(a-b)^2+ r_1(\sqrt{ab}-b)^2+\mathrm{K}(\sqrt{h}, 2)^{\hat{r}_1}\left(a^{1-v}b^v\right)^2.
\end{equation}
%where $h=\dfrac{b}{a}$, $r=\min\{v, 1-v\}$, $r_1=\min\{2r, 1-2r\}$ and $\hat{r}_1=\min\{2r_1, 1-2r_1\}$.
%where $h=\dfrac{b}{a}$, $r_1=\min\{4v, 1-4v\}$, $r_2=\min\{2-4v,4v-1\}$, $r_3=\min\{4v-2, 3-4v\}$ and $r_4=\min\{4-4v,4v-3\}$ .
\end{corollary}

\begin{proof}
Replacing $a$ by $a^2$ and $b$ by $b^2$ in \eqref{th211} and \eqref{th212}, respectively, we have
\[
(1-v)a^2+vb^2\geq v(a-b)^2+r_1(\sqrt{ab}-a)^2+\mathrm{K}(\sqrt{h}, 2)^{\hat{r}_1}\left(a^{1-v}b^v\right)^2
\]
and
\begin{equation}
\label{co213}
(1-v)a^2+vb^2\geq (1-v)(a-b)^2+ r_1(\sqrt{ab}-b)^2+\mathrm{K}(\sqrt{h}, 2)^{\hat{r}_1}\left(a^{1-v}b^v\right)^2.
\end{equation}
If $0<v\leq \frac{1}{2}$, then by the first inequality above, we obtain
\begin{align*}
((1-v)a+vb)^2&-v^2(a-b)^2\\
&=(1-v)a^2+vb^2-v(a-b)^2\\
&\geq r_1(\sqrt{ab}-a)^2+\mathrm{K}(\sqrt{h}, 2)^{\hat{r}_1}\left(a^{1-v}b^v\right)^2.
\end{align*}
If $\frac{1}{2}<v<1$, then by using \eqref{co213}, we get
\begin{align*}
((1-v)a+vb)^2&-(1-v)^2(a-b)^2\\
&=(1-v)a^2+vb^2-(1-v)(a-b)^2\\
&\geq r_1(\sqrt{ab}-b)^2+\mathrm{K}(\sqrt{h}, 2)^{\hat{r}_1}\left(a^{1-v}b^v\right)^2.
\end{align*}
\end{proof}

\begin{theorem}\label{th22}
 Let $a, b>0$ and $v\in(0, 1)$. Then

$\rm (I)$ If $0<v\leq \frac{1}{2}$, then
\begin{equation}
\label{th221}
  (1-v)a+vb\leq (1-v)(\sqrt{a}-\sqrt{b})^2-r_1(\sqrt[4]{ab}-\sqrt{b})^2+\mathrm{K}(\sqrt[4]{h}, 2)^{-\hat{r}_1}a^{1-v}b^v,
\end{equation}

$\rm (II)$ If $\frac{1}{2}<v<1$, then
\begin{equation}
\label{th222}
  (1-v)a+vb\leq v(\sqrt{a}-\sqrt{b})^2- r_1(\sqrt[4]{ab}-\sqrt{a})^2+\mathrm{K}(\sqrt[4]{h}, 2)^{-\hat{r}_1}a^{1-v}b^v,
\end{equation}
where $h=\dfrac{b}{a}$, $r=\min\{v, 1-v\}$, $r_1=\min\{2r, 1-2r\}$ and $\hat{r}_1=\min\{2r_1, 1-2r_1\}$.
%where $h=\dfrac{b}{a}$, $r_1=\min\{4v, 1-4v\}$, $r_2=\min\{2-4v,4v-1\}$, $r_3=\min\{4v-2, 3-4v\}$ and $r_4=\min\{4-4v,4v-3\}$ .
\end{theorem}

\begin{proof}
The proof of inequality \eqref{th222} is similar to that of \eqref{th221}. Thus, we only need to prove \eqref{th221}.

If $0<v<\frac{1}{2}$, then by the inequality \eqref{wu1}, we deduce
\begin{align*}
  &\mathrm{K}(\sqrt[4]{h}, 2)^{-\hat{r}_1}a^{1-v}b^v+(1-v)(\sqrt{a}-\sqrt{b})^2-(1-v)a-vb\\
  &=\mathrm{K}(\sqrt[4]{h}, 2)^{-\hat{r}_1}a^{1-v}b^v+(1-2v)b+2v\sqrt{ab}-2\sqrt{ab}\\
  &\geq \mathrm{K}(\sqrt[4]{h}, 2)^{-\hat{r}_1}a^{1-v}b^v+K(\sqrt[4]{h}, 2)^{\hat{r}_1}a^v b^{1-v}+r_1(\sqrt[4]{ab}-\sqrt{b})^2-2\sqrt{ab}\\
  &\geq r_1(\sqrt[4]{ab}-\sqrt{b})^2.
\end{align*}

If $v=\frac{1}{2}$, the inequality \eqref{th221} becomes equality.
\end{proof}
%$$(1-u)a+ub+r(\sqrt[4]{ab}-\sqrt{b})^2\leq R(\sqrt{a}-\sqrt{b})^2+\mathrm{K}(\sqrt[4]{h}, 2)^{-r'}a^{1-u}b^u
%$$
%$$((1-u)a+ub)^2+r^2(\sqrt{ab}-b)^2\leq R^2(a-b)^2+\mathrm{K}(\sqrt{h}, 2)^{-r'}(a^{1-u}b^u)^2, u\in [\frac{1}{2}, 1]$$

\begin{remark}
By the properties of Kantorovich constant, \eqref{th221} and \eqref{th222} are better than \eqref{p2} and \eqref{p4}, respectively.
As a direct consequence of Theorem \ref{th22}, we have the following inequalities with respect to the Heinz mean:
%If $0<v\leq \frac{1}{2}$, then
\begin{equation}
\label{rem2}
  \frac{a+b}{2}\leq R(\sqrt{a}-\sqrt{b})^2-\frac{1}{2}r_1[(\sqrt[4]{ab}-\sqrt{a})^2+(\sqrt[4]{ab}-\sqrt{b})^2]+\mathrm{K}(\sqrt[4]{h}, 2)^{-\hat{r}_1}H_v(a,b),
\end{equation}
where $R=\max\{v, 1-v\}$.
%If $\frac{1}{2}<v< 1$, then
%\begin{equation}
%\label{eqlem}
% \frac{a+b}{2}\leq v(\sqrt{a}-\sqrt{b})^2-\frac{1}{2}r_2[(\sqrt[4]{ab}-\sqrt{a})^2+(\sqrt[4]{ab}-\sqrt{b})^2]+\mathrm{K}(\sqrt[4]{h}, 2)^{-\hat{r}_2}H_v(a,b),
%\end{equation}
\end{remark}

\begin{corollary}\label{co22}
 Let $a, b>0$ and $v\in(0, 1)$. Then

$\rm (I)$ If $0<v\leq \frac{1}{2}$, then
\[
  ((1-v)a+vb)^2\leq (1-v)^2(a-b)^2-r_1(\sqrt{ab}-b)^2+\mathrm{K}(\sqrt{h}, 2)^{-\hat{r}_1}\left(a^{1-v}b^v\right)^2.
\]

$\rm (II)$ If $\frac{1}{2}<v<1$, then
\[
  ((1-v)a+vb)^2\leq v^2(a-b)^2- r_1(\sqrt{ab}-a)^2+\mathrm{K}(\sqrt{h}, 2)^{-\hat{r}_1}\left(a^{1-v}b^v\right)^2.
\]
%where $h=\dfrac{b}{a}$, $r=\min\{v, 1-v\}$, $r_1=\min\{2r, 1-2r\}$ and $\hat{r}_1=\min\{2r_1, 1-2r_1\}$.
%where $h=\dfrac{b}{a}$, $r_1=\min\{4v, 1-4v\}$, $r_2=\min\{2-4v,4v-1\}$, $r_3=\min\{4v-2, 3-4v\}$ and $r_4=\min\{4-4v,4v-3\}$ .
\end{corollary}
\begin{proof}
Replacing $a$ by $a^2$ and $b$ by $b^2$ in \eqref{th221} and \eqref{th222}, respectively, we have
\[
 (1-v)a^2+vb^2\leq (1-v)(a-b)^2-r_1(\sqrt{ab}-b)^2+\mathrm{K}(\sqrt{h}, 2)^{-\hat{r}_1}\left(a^{1-v}b^v\right)^2
\]
and
\[
(1-v)a^2+vb^2\leq v(a-b)^2- r_1(\sqrt{ab}-a)^2+\mathrm{K}(\sqrt{h}, 2)^{-\hat{r}_1}\left(a^{1-v}b^v\right)^2.
\]
The rest proof is similar to Corollary \ref{co21}.
%If $0<v\leq \frac{1}{2}$, then we get
%\begin{align*}
%((1-v)a+vb)^2&-(1-v)^2(a-b)^2\\
%&=(1-v)a^2+vb^2-(1-v)(a-b)^2\\
%&\le -r_1(\sqrt{ab}-b)^2+\mathrm{K}(\sqrt[4]{h}, 2)^{\hat{r}_1}\left(a^{1-v}b^v\right)^2.
%\end{align*}
%If $\frac{1}{2}<v<1 $, then we get
%\begin{align*}
%((1-v)a+vb)^2&-v^2(a-b)^2\\
%&=(1-v)a^2+vb^2-v(a-b)^2\\
%&\le -r_2(\sqrt{ab}-b)^2+\mathrm{K}(\sqrt[4]{h}, 2)^{\hat{r}_2}\left(a^{1-v}b^v\right)^2.
%\end{align*}
\end{proof}

\section{Operator inequalities}

If $A$ is a selfadjoint operator and $f$ is a real valued continuous function on $\mathrm{Sp}(A)$ (the spectrum of $A$), then $f(t)\ge 0$
for every $t\in \mathrm{Sp}(A)$ implies that $f(A)\ge 0$, i.e., $f(A)$ is a positive operator on $\mathcal{H}$. Equivalently,
if both $f$ and $g$ are real valued continuous functions on $\mathrm{Sp}(A)$, then the following monotonic property
of operator functions holds:
 \[f(t)\ge g(t)~\mathrm{for~any}~t\in \mathrm{Sp}(A)~\mathrm{implies~that}~f(A)\ge g(A)\]
in the operator order of $\mathcal{B}(\mathcal{H})$.

\begin{theorem}\label{th31}
Let $A,B\in {\mathcal{B}}^{++}(\mathcal{H})$ and positive real numbers $m, m', M, M'$ satisfy either $0 < m'I \le A \le mI < MI \le B \le M'I$ or $0 < m'I \le B \le mI < MI \le A \le M'I$.

  $\rm (I)$ If $0<v\leq \frac{1}{2}$, then
   \begin{equation}
       \label{th311}
        A\nabla_v B\ge 2v(A\nabla B-A\#B)+r_1(A\#B-2A\#_{\frac{1}{4}}B+A)+\mathrm{K}(\sqrt[4]{h}, 2)^{\hat{r}_1}A\#_vB,
   \end{equation}

  $\rm (II)$ if $\frac{1}{2}<v< 1$, then
  \begin{equation}
       \label{th312}
         A\nabla_v B\ge 2(1-v)(A\nabla B-A\#B)+r_1(A\#B-2A\#_{\frac{3}{4}}B+B)+\mathrm{K}(\sqrt[4]{h}, 2)^{\hat{r}_1}A\#_vB,
  \end{equation}
%\[(1-v)a+vb\leq (1-v)(\sqrt{a}-\sqrt{b})^2+ r_2(\sqrt[4]{ab}-\sqrt{b})^2+K(\sqrt[4]{h}, 2)^{R_2}a^{1-v}b^v,\]
where $h=\frac{M}{m}$, $r=\min\{v, 1-v\}$, $r_1=\min\{2r, 1-2r\}$ and $\hat{r}_1=\min\{2r_1, 1-2r_1\}$.
Equality holds if and only if $A=B$ and $m=M$.
\end{theorem}

 \begin{proof}

If $0<v\leq \frac{1}{2}$, it follows from the inequality \eqref{th211} that for any $x>0$,
\begin{equation}
       \label{th313}
(1-v)+vx\geq v(\sqrt{x}-1)^2+r_1(\sqrt[4]{x}-1)^2+\mathrm{K}(\sqrt[4]{x}, 2)^{\hat{r}_1}x^v.
 \end{equation}

Putting $X=A^{-\frac{1}{2}}BA^{-\frac{1}{2}},$ under the condition $0 < m'I \le A \le mI < MI \le B \le M'I$, we have
$$I\le hI =\frac{M}{m}I \le X\le h'I =\frac{M'}{m'}I,$$
 and then $\mathrm{Sp}(X)\subset [h, h']\subset(1,+\infty)$. Thus for  positive operator $X$,
 it can be deduced from the inequality \eqref{th313} and the monotonic property of operator functions that
\[
 (1-v)I+vX\geq v(X-2X^{\frac{1}{2}}+I)+r_1(X^{\frac{1}{2}}-2X^{\frac{1}{4}}+I)+\mathop {\min }\limits_{h\le x\le h'}{\rm K}(\sqrt[4]{x}, 2)^{\hat{r}_1}x^v.
\]
On the other hand, the Kantorovich constant ${\rm K}(t,2)$ is an increasing function on $(1,+\infty)$, we get
\begin{equation}
       \label{th314}
   \begin{split}
  (1-v)I+vA^{-\frac{1}{2}}BA^{-\frac{1}{2}}
 \geq& v(A^{-\frac{1}{2}}BA^{-\frac{1}{2}}-2(A^{-\frac{1}{2}}BA^{-\frac{1}{2}})^{\frac{1}{2}}+I)\\
 &+r_1((A^{-\frac{1}{2}}BA^{-\frac{1}{2}})^{\frac{1}{2}}-2(A^{-\frac{1}{2}}BA^{-\frac{1}{2}})^{\frac{1}{4}}+I)\\
  &+ {\rm K}(\sqrt[4]{h}, 2)^{\hat{r}_1}(A^{-\frac{1}{2}}BA^{-\frac{1}{2}})^v,
\end{split}
 \end{equation}
%\[
%\left( {1-\mu } \right)I+\mu A^{-\frac{1}{2}}BA^{-\frac{1}{2}}\le {\rm
%K}(h,2)\left( {(1-\mu )I+\mu \left(A^{-\frac{1}{2}}BA^{-\frac{1}{2}}\right)^{-1}}
%\right)^{-1}.
%\]

Likewise, under the condition $0 < m'I \le B \le mI < MI \le A \le M'I$, we have $0\le \frac{1}{h'}I \le X\le  \frac{1}{h}I<I$
and then $\mathrm{Sp}(X)\subset [ \frac{1}{h'},  \frac{1}{h}]\subset(0,1)$.
Thus for  positive operator $X$, we obtain
\[
 (1-v)I+vX\geq v(X-2X^{\frac{1}{2}}+I)+r_1(X^{\frac{1}{2}}-2X^{\frac{1}{4}}+I)+\mathop {\min }\limits_{ \frac{1}{h'}\le x\le \frac{1}{h}}{\rm K}(\sqrt[4]{x}, 2)^{\hat{r}_1}x^v.
\]
On the other hand, the Kantorovich constant ${\rm K}(t,2)$ is an decreasing function on $(0,1)$ and ${\rm K}(\frac{1}{t},2)={\rm K}(t,2)$,we get
\begin{equation}
       \label{th315}
 (1-v)I+vX\geq v(X-2X^{\frac{1}{2}}+I)+r_1(X^{\frac{1}{2}}-2X^{\frac{1}{4}}+I)+{\rm K}(\sqrt[4]{h}, 2)^{\hat{r}_1}x^v.
 \end{equation}
It is striking that we obtain two same inequalities \eqref{th314} and \eqref{th315} under the
two different condition. Then multiplying inequality \eqref{th314} or \eqref{th315}
by $A^{\frac{1}{2}}$ on both sides, we can deduce the required inequality (\ref{th311}).

If $\frac{1}{2}<v< 1$, the inequality \eqref{th312} follows from inequality \eqref{th212}
by the similar method.
\end{proof}

The operator version of \eqref{rem1} can be shown as
\begin{corollary}
Under the same conditions as Theorem \ref{th31}, then
%Let $A,B\in {\mathcal{B}}^{++}(\mathcal{H})$ and positive real numbers $m, m', M, M'$
%satisfy either $0 < m'I \le A \le mI < MI \le B \le M'I$ or $0 < m'I \le B \le mI < MI \le A \le M'I$.
%
%  $\rm (I)$ If $0<v\leq \frac{1}{2}$, then
   \begin{equation}
       \label{co31}
        A\nabla B\geq 2r(A\nabla B-A\#B)+r_1(A\nabla B+A\#B-2H_{\frac{1}{4}}(A,B))+\mathrm{K}(\sqrt[4]{h}, 2)^{\hat{r}_1}H_v(A,B).
  \end{equation}

  %$\rm (II)$ If $\frac{1}{2}<v< 1$, then
%  \begin{equation}
%       \label{co312}
%          A\nabla B\geq 2(1-v)(A\nabla B-A\#B)+r_2\left(A\nabla B+A\#B-2H_{\frac{1}{4}}(A,B)\right)+\mathrm{K}(\sqrt[4]{h}, 2)^{\hat{r}_2}H_v(A,B)
%  \end{equation}
\end{corollary}

 %\begin{proof}
% By Lemma (), utilizing the same ideas as in the prove of Theorem 3.1, we can obtain the required results.
% \end{proof}

\begin{theorem}\label{th32}
Let $A,B\in {\mathcal{B}}^{++}(\mathcal{H})$ and positive real numbers $m, m', M, M'$ satisfy either $0 < m'I \le A \le mI < MI \le B \le M'I$ or $0 < m'I \le B \le mI < MI \le A \le M'I$.

  $\rm (I)$ If $0<v\leq \frac{1}{2}$, then
   \begin{equation}
       \label{th321}
        A\nabla_v B\leq 2(1-v)(A\nabla B-A\#B)-r_1(A\#B-2A\#_{\frac{3}{4}}B+B)+\mathrm{K}(\sqrt[4]{h}, 2)^{-\hat{r}_1}A\#_vB,
   \end{equation}

  $\rm (II)$ if $\frac{1}{2}<v< 1$, then
  \begin{equation}
       \label{th322}
         A\nabla_v B\leq 2v(A\nabla B-A\#B)-r_1(A\#B-2A\#_{\frac{1}{4}}B+A)+\mathrm{K}(\sqrt[4]{h}, 2)^{-\hat{r}_1}A\#_vB,
  \end{equation}
%\[(1-v)a+vb\leq (1-v)(\sqrt{a}-\sqrt{b})^2+ r_2(\sqrt[4]{ab}-\sqrt{b})^2+K(\sqrt[4]{h}, 2)^{R_2}a^{1-v}b^v,\]
where $h=\dfrac{M}{m}$, $r=\min\{v, 1-v\}$, $r_1=\min\{2r, 1-2r\}$ and $\hat{r}_1=\min\{2r_1, 1-2r_1\}$.
Equality holds if and only if $A=B$ and $m=M$.
\end{theorem}
 \begin{proof}
By \eqref{th221} and \eqref{th222}, using the same ideas as the proof of Theorem \ref{th31}, we can get this theorem.
\end{proof}

The operator version of \eqref{rem2} can be shown as
\begin{corollary}
Under the same conditions as Theorem \ref{th32}, then
 % $\rm (I)$ If $0<v\leq \frac{1}{2}$, then
   \begin{equation}
       \label{co32}
        A\nabla B\leq 2R(A\nabla B-A\#B)-r_1(A\nabla B+A\#B-2H_{\frac{1}{4}}(A,B))+\mathrm{K}(\sqrt[4]{h}, 2)^{-\hat{r}_1}H_v(A,B),
    \end{equation}
   where $R=\max\{v, 1-v\}$.
  %$\rm (II)$ If $\frac{1}{2}<v< 1$, then
%  \begin{equation}
%       \label{eqlem}
%          A\nabla B\leq 2v(A\nabla B-A\#B)-r_2\left(A\nabla B+A\#B-2H_{\frac{1}{4}}(A,B)\right)+\mathrm{K}(\sqrt[4]{h}, 2)^{-\hat{r}_2}H_v(A,B)
%  \end{equation}
\end{corollary}

\begin{remark}
\eqref{co31} and \eqref{co32} are sharper than $(3.4)$ in \cite{Kittaneh2}.

If $0<v\leq \frac{1}{2}$, combining  \eqref{th311} with \eqref{th321}, then
\begin{align*}
0&\le A\#_vB\\
&\le 2v(A\nabla B-A\#B)+A\#_vB\\
&\le 2v(A\nabla B-A\#B)+r_1(A\#B-2A\#_{\frac{1}{4}}B+A)+A\#_vB\\
&\le 2v(A\nabla B-A\#B)+r_1(A\#B-2A\#_{\frac{1}{4}}B+A)+\mathrm{K}(\sqrt[4]{h}, 2)^{\hat{r}_1}A\#_vB\\
&\le A\nabla_v B\\
&\le 2(1-v)(A\nabla B-A\#B)-r_1(A\#B-2A\#_{\frac{3}{4}}B+B)+\mathrm{K}(\sqrt[4]{h}, 2)^{-\hat{r}_1}A\#_vB\\
&\le 2(1-v)(A\nabla B-A\#B)-r_1(A\#B-2A\#_{\frac{3}{4}}B+B)+A\#_vB\\
&\le 2(1-v)(A\nabla B-A\#B)+A\#_vB.
\end{align*}
By the properties of Kantorovich constant, \eqref{th311} and \eqref{th312} are better than \eqref{p5} and \eqref{p7}, respectively.

If $\frac{1}{2}<v< 1$, combining \eqref{th312} with \eqref{th322}, then
\begin{align*}
0&\le A\#_vB\\
&\le 2(1-v)(A\nabla B-A\#B)+A\#_vB\\
&\le 2(1-v)(A\nabla B-A\#B)+r_1(A\#B-2A\#_{\frac{3}{4}}B+B)+A\#_vB\\
&\le 2(1-v)(A\nabla B-A\#B)+r_1(A\#B-2A\#_{\frac{3}{4}}B+B)+\mathrm{K}(\sqrt[4]{h}, 2)^{\hat{r}_1}A\#_vB\\
&\le A\nabla_v B\\
&\le 2v(A\nabla B-A\#B)-r_1(A\#B-2A\#_{\frac{1}{4}}B+A)+\mathrm{K}(\sqrt[4]{h}, 2)^{-\hat{r}_1}A\#_vB\\
&\le 2v(A\nabla B-A\#B)-r_1(A\#B-2A\#_{\frac{1}{4}}B+A)+A\#_vB\\
&\le 2v(A\nabla B-A\#B)+A\#_vB.
\end{align*}
\eqref{th321} and \eqref{th322} are better than \eqref{p6} and \eqref{p8}, respectively.
\end{remark}

\section{Hilbert-Schmidt norm inequalities}
In this section, we present the improved Young and Heinz inequalities for the Hilbert-Schmidt norm.

Let $M_n(\mathbb{C})$ denote the algebra of all
$n\times n$ complex matrices. The Hilbert-Schmidt norm of $A\in M_n(\mathbb{C})$ is denoted by
$\left\| A\right\|_F^2$. It is well-known that the Hilbert-Schmidt norm is unitarily invariant in
the sense that $\|UAV\|_F^2=\|A\|_F^2$ for all unitary matrices $U,V \in M_n(\mathbb{C})$ (See \cite[p.341-342]{Horn}).
The spectrum of $A$ and $B$ are denoted by $\mathrm{Sp}(A)=\{\lambda _1 ,\lambda _2 , \cdots ,\lambda _n\}$
and $\mathrm{Sp}(B)=\{\nu _1 ,\nu _2 , \cdots ,\nu _n\}$, respectively.
The Schur product (Hadamard product) of two matrices $A, B\in M_n(\mathbb{C})$ is the entrywise product and denoted by $A\circ B$.

%Based on the refined Young inequality (1.7) in \cite{Kittaneh1} and the reverse Young inequality (2.4) in \cite{Kittaneh2},
 Hirzallah and Kittaneh \cite{Hirzallah} and Kittaneh and Manasrah \cite{Kittaneh2} have showed that if $A,B,X \in M_n(\mathbb{C})$ with $A$ and $B$ positive semidefinite matrices and $v \in [0,1]$, then
\begin{equation}
\label{eq41}
r^2\left\| {AX - XB} \right\|_F^2\le \left\| {(1 - v )AX + v XB} \right\|_F^2-\left\| {A^{1 - v }XB^v
} \right\|_F^2 \le  R^2\left\| {AX - XB} \right\|_F^2 ,
\end{equation}
where $r=\min\{v,1-v\}$, $R=\max\{v,1-v\}$.

% Kittaneh and Manasrah \cite{Kittaneh1, Kittaneh2} established the improved Heinz norm inequalities as follows
%\begin{equation}
%\label{eq42}
%2r\left\| {AX - XB} \right\|_F^2\le\left\| {AX + XB} \right\|_F^2-\left\| {A^{1 - v }XB^v + A^v XB^{1
%- v }} \right\|_F^2  \le  2R\left\| {AX - XB} \right\|_F^2.
%\end{equation}

Applying Corollary \ref{co21} and \ref{co22}, we derive two theorems which improve \eqref{eq41}.

\begin{theorem}
Suppose $A,B,X \in M_n(\mathbb{C})$ such that $A$ and $B$ are
two positive definite matrices.
%and satisfy either $0 < m'I \le A \le mI < MI \le B \le M'I$ or $0 < m'I \le B \le mI < MI \le A \le M'I$,
%where $I$ represents the identity matrix and $m, m', M, M'\in {\rm R}$.
 Let\[\mathrm{\underline{K}}=\min\left\{\mathrm{K}(\left({\lambda _i }/{\nu _j }\right)^{\frac{1}{2}}, 2),i,j=1,2,\cdots, n\right\}.\]

  $\rm (I)$ If $0<v\leq \frac{1}{2}$, then
  \begin{equation}
    \label{th411}
    \begin{split}
       \left\| {(1 - v )AX + v XB} \right\|_F^2&-v^2\left\| {AX - XB} \right\|_F^2 \\
       &\ge r_1\left\| {A^{\frac{1}{2}}XB^{\frac{1}{2}} - AX} \right\|_F^2+\mathrm{\underline{K}}^{\hat{r}_1}\left\|
       {A^{1 - v }XB^v } \right\|_F^2,
    \end{split}
  \end{equation}

 $\rm (II)$ if $\frac{1}{2}<v< 1$, then
 \begin{equation}
    \label{th412}
    \begin{split}
     \left\| {(1 - v )AX + v XB} \right\|_F^2& -(1-v)^2\left\| {AX - XB} \right\|_F^2\\
     &\ge r_1\left\| A^{\frac{1}{2}}XB^{\frac{1}{2}} - XB \right\|_F^2+\mathrm{\underline{K}}^{\hat{r}_1}\left\|
     {A^{1 - v }XB^v } \right\|_F^2,
  \end{split}
 \end{equation}
where $r=\min\{v, 1-v\}$, $r_1=\min\{2r, 1-2r\}$ and $\hat{r}_1=\min\{2r_1, 1-2r_1\}$.
\end{theorem}

\begin{proof} Since $A$ and $B$ are positive definite, it follows by the
spectral theorem that there exist unitary matrices $U,V \in M_n(\mathbb{C})$ such that
\[
A = U\Lambda _1 U^ * ,
B = V\Lambda _2 V^ * ,
\]

\noindent
where $\Lambda _1 = \mathrm{diag}(\lambda _1 ,\lambda _2 , \cdots ,\lambda _n
),$ $\Lambda _2 = \mathrm{diag}(\nu _1 ,\nu _2 , \cdots ,\nu _n ),$ $\lambda _i , \nu _i >
0,$ $i = 1,2, \cdots ,n.$

Let $Y = U^ * XV = [y_{ij} ]$, then
\begin{align*}
(1 - v )AX + v XB &= U((1 - v )\Lambda _1 Y + v Y\Lambda _2 )V^ *\\
&=U[((1 - v )\lambda _i + v \nu _j )\circ Y  ]V^ * ,
\end{align*}
$$AX - XB = U[(\lambda _i - \nu _j )\circ Y  ]V^ * ,$$
$$A^{\frac{1}{2}}XB^{\frac{1}{2}} - AX=U[((\lambda _i \nu _j)^{\frac{1}{2}}- \lambda _i)\circ Y  ]V^ *, $$
$$A^{\frac{1}{2}}XB^{\frac{1}{2}} - XB=U[((\lambda _i \nu _j)^{\frac{1}{2}}-  \nu _j)\circ Y  ]V^ * $$
and
$$A^{1 - v }XB^v =
U((\lambda _i^{1 - v } \nu _j^v )\circ Y )V^ * .$$

If $0<v\leq \frac{1}{2}$, utilizing the inequality \eqref{co211} and the unitary invariance of the Hilbert-Schmidt norm, we have
\begin{align*}
\left\| {(1 - v )AX + v XB} \right\|_F^2&-v^2\left\| {AX - XB} \right\|_F^2 \\
&= \sum\limits_{i,j = 1}^n {((1- v )\lambda _i + v \nu _j )^2\vert y_{ij} \vert ^2}-v^2\sum\limits_{i,j = 1}^n(\lambda _i - \nu _j )^2\vert y_{ij}\vert^2 \\
&= \sum\limits_{i,j = 1}^n \left[{((1- v )\lambda _i + v \nu _j )^2\vert y_{ij} \vert ^2}-v^2(\lambda _i - \nu _j )^2\vert y_{ij}\vert^2\right] \\
&\ge\sum\limits_{i,j = 1}^n\left[ r_1((\lambda _i \nu _j)^{\frac{1}{2}}- \lambda _i)^2\vert y_{ij} \vert ^2+\mathrm{K}\left(\left({\lambda _i }/{\nu _j }\right)^{\frac{1}{2}}, 2\right)^{\hat{r}_1}\left(\lambda _i^{1 - v } \nu _j^v\right)^2\vert y_{ij} \vert ^2\right]\\
&\ge r_1\sum\limits_{i,j = 1}^n ((\lambda _i \nu _j)^{\frac{1}{2}}- \lambda _i)^2\vert y_{ij} \vert ^2+\mathrm{\underline{K}}^{\hat{r}_1}\sum\limits_{i,j = 1}^n\left(\lambda _i^{1 - v } \nu _j^v\right)^2\vert y_{ij} \vert ^2\\
&=r_1\left\| {A^{\frac{1}{2}}XB^{\frac{1}{2}} - AX} \right\|_F^2+\mathrm{\underline{K}}^{\hat{r}_1}\left\|
       {A^{1 - v }XB^v } \right\|_F^2.
 \end{align*}
Similarly, if $\frac{1}{2}<v< 1$, using the inequality \eqref{co212}, we can derive \eqref{th412}.
\end{proof}

\begin{theorem}
Suppose $A,B,X \in M_n(\mathbb{C})$ such that $A$ and $B$ are
two positive definite matrices. Let
\[\mathrm{\underline{K}}=\min\left\{\mathrm{K}(\left({\lambda _i }/{\nu _j }\right)^{\frac{1}{2}}, 2),i,j=1,2,\cdots, n\right\}.\]

$(I)$ If $0<v\leq \frac{1}{2}$,
  \begin{equation}
    \label{th421}
    \begin{split}
       \left\| {(1 - v )AX+ v XB} \right\|_F^2&-(1-v)^2\left\| {AX - XB} \right\|_F^2 \\
       &\le\mathrm{\underline{K}}^{-\hat{r}_1}\left\|
       {A^{1 - v }XB^v } \right\|_F^2- r_1\left\| {A^{\frac{1}{2}}XB^{\frac{1}{2}} - XB} \right\|_F^2,
    \end{split}
  \end{equation}

 $\rm (II)$ if $\frac{1}{2}<v< 1$, then
 \begin{equation}
    \label{th422}
    \begin{split}
     \left\| {(1 - v )AX + v XB} \right\|_F^2&- v^2\left\| {AX - XB} \right\|_F^2\\
     &\le\mathrm{\underline{K}}^{-\hat{r}_1}\left\|
     {A^{1 - v }XB^v } \right\|_F^2-r_1\left\| A^{\frac{1}{2}}XB^{\frac{1}{2}} - AX \right\|_F^2,
    \end{split}
 \end{equation}
where $r=\min\{v, 1-v\}$, $r_1=\min\{2r, 1-2r\}$ and $\hat{r}_1=\min\{2r_1, 1-2r_1\}$.
\end{theorem}

\begin{proof}
By using the same ideas as the prove of Theorem 4.1 and Corollary \ref{co22},
we can obtain the required results.
\end{proof}

\begin{remark} If $0<v\leq \frac{1}{2}$, combining  \eqref{th411}and \eqref{th421}
with \eqref{eq41}, we obtain
  \begin{align*}
0\le&  \left\| {A^{1 - v }XB^v } \right\|_F^2\\
\le&  v^2\left\| {AX - XB} \right\|_F^2 +\left\|
       {A^{1 - v }XB^v } \right\|_F^2\\
\le&  v^2\left\| {AX - XB} \right\|_F^2 +\mathrm{\underline{K}}^{\hat{r}_1}\left\|
       {A^{1 - v }XB^v } \right\|_F^2\\
\le&  v^2\left\| {AX - XB} \right\|_F^2 +r_1\left\| {A^{\frac{1}{2}}XB^{\frac{1}{2}} - AX} \right\|_F^2+\mathrm{\underline{K}}^{\hat{r}_1}\left\|
       {A^{1 - v }XB^v } \right\|_F^2\\
\le& \left\| {(1 - v )AX+ v XB} \right\|_F^2\\
\le& (1-v)^2\left\| {AX - XB} \right\|_F^2- r_1\left\| {A^{\frac{1}{2}}XB^{\frac{1}{2}} - XB} \right\|_F^2 +\mathrm{\underline{K}}^{-\hat{r}_1}\left\|
       {A^{1 - v }XB^v } \right\|_F^2\\
\le& (1-v)^2\left\| {AX - XB} \right\|_F^2+\mathrm{\underline{K}}^{-\hat{r}_1}\left\|
       {A^{1 - v }XB^v } \right\|_F^2\\
\le& (1-v)^2\left\| {AX - XB} \right\|_F^2+\left\|
       {A^{1 - v }XB^v } \right\|_F^2,
  \end{align*}
  if $\frac{1}{2}<v< 1$, combining \eqref{th412}and \eqref{th422} with \eqref{eq41},
  we can obtain similar results,
\begin{align*}
0\le&  \left\| {A^{1 - v }XB^v } \right\|_F^2\\
\le&  (1-v)^2\left\| {AX - XB} \right\|_F^2 +\left\|
       {A^{1 - v }XB^v } \right\|_F^2\\
\le&  (1-v)^2\left\| {AX - XB} \right\|_F^2 +\mathrm{\underline{K}}^{\hat{r}_1}\left\|
       {A^{1 - v }XB^v } \right\|_F^2\\
\le&  (1-v)^2\left\| {AX - XB} \right\|_F^2 +r_1\left\| {A^{\frac{1}{2}}XB^{\frac{1}{2}} - XB} \right\|_F^2+\mathrm{\underline{K}}^{\hat{r}_1}\left\|
       {A^{1 - v }XB^v } \right\|_F^2\\
\le& \left\| {(1 - v )AX+ v XB} \right\|_F^2\\
\le& v^2\left\| {AX - XB} \right\|_F^2- r_1\left\| {A^{\frac{1}{2}}XB^{\frac{1}{2}} - AX} \right\|_F^2 +\mathrm{\underline{K}}^{-\hat{r}_1}\left\|
       {A^{1 - v }XB^v } \right\|_F^2\\
\le& v^2\left\| {AX - XB} \right\|_F^2+\mathrm{\underline{K}}^{-\hat{r}_1}\left\|
       {A^{1 - v }XB^v } \right\|_F^2\\
 \le& v^2\left\| {AX - XB} \right\|_F^2+\left\|
       {A^{1 - v }XB^v } \right\|_F^2,
  \end{align*}
which are improvements of $\left\| A^{1-v}XB^v\right\|_F^2\le\left\| {(1 - v )AX+ v XB} \right\|_F^2$ (See \cite{Bhatia1, Kosaki}).
\end{remark}

\end{document}